\documentclass[letterpaper,11pt,leqno]{amsart}
\usepackage{amssymb}
\usepackage{epsfig}
\usepackage{graphicx}
\usepackage{color}
\usepackage{pdfsync}
\usepackage{amsmath}
\usepackage{setspace}
\usepackage{hyperref}
\usepackage{geometry}
\usepackage{amsmath}
\usepackage{amssymb}
\usepackage{enumitem}%Pour réduire marges de  \item

\usepackage[latin9]{inputenc}

\newcommand{\R}{\mathbb{R}}

\newcommand{\N}{\mathbb{N}}
\newcommand{\dom}{\mathop\mathrm{\rm dom}}

\newcommand{\argmin}{\mbox{argmin}\;}
\newcommand{\prox}{\mbox{prox}}

\newcommand{\eps}{\varepsilon}

\newcommand{\hone}{\mathbf{H}_1}
\newcommand{\htwo}{\mathbf{H}_2}
\newcommand{\hthree}{\mathbf{H}_3}

\newcommand{\norm}[1]{\|#1\|}

\newcommand{\kin}{_{k\in\N}}
\newcommand{\bx}{}
\newcommand{\Rcupinf}{\R\cup\{+\infty\}}

\newcommand{\HH}{\mathcal{H}}

\newcommand{\Sym}{\mathcal{S}_{\text{\begin{tiny}++\end{tiny}}}}
\newcommand{\sym}{\mathcal{S}_{\text{\begin{tiny}+\end{tiny}}}}
\newcommand{\proj}{\mbox{\rm proj\,}}
\newcommand{\rank}{\mbox{\rm rank\,}}

\DeclareFontEncoding{FMS}{}{}
\DeclareFontSubstitution{FMS}{futm}{m}{n}
\DeclareFontEncoding{FMX}{}{}
\DeclareFontSubstitution{FMX}{futm}{m}{n}
\DeclareSymbolFont{fouriersymbols}{FMS}{futm}{m}{n}
\DeclareSymbolFont{fourierlargesymbols}{FMX}{futm}{m}{n}
\DeclareMathDelimiter{\VERT}{\mathord}{fouriersymbols}{152}{fourierlargesymbols}{147}

\newtheorem{lemma}{Lemma}
\newtheorem{example}{Example}
\newtheorem{theorem}{Theorem}
\newtheorem{remark}{Remark}
\newtheorem{proposition}{Proposition}

\begin{document}
\pagestyle{empty}

\title{\large{Splitting methods with variable metric for K{\L}  functions  and general convergence rates}}

\author{Pierre Frankel, Guillaume Garrigos, Juan Peypouquet\\ }

%\thanks{{\tiny P. Frankel (p.frankel30@orange.fr) \& G. Garrigos (guillaume.garrigos@gmail.com)\\Institut de Math\'ematiques et Mod\'elisation de Montpellier, UMR 5149 CNRS.\\Universit\'e Montpellier 2, Place Eug\`ene Bataillon, 34095 Montpellier cedex 5, France. \\G. Garrigos \& J. Peypouquet (juan.peypouquet@usm.cl)\\Departamento de Matem\'atica \& AM2V.\\Universidad T\'ecnica Federico Santa Mar\'\i a, Avenida  Espa\~na 1680, Valpara\'\i so, Chile.}}

%\date{\today}

\begin{abstract}
{\small We study the convergence of general abstract descent methods applied to a lower semicontinuous nonconvex function $f$ that satisfies the Kurdyka-\L ojasiewicz inequality in a Hilbert space. We prove that any precompact sequence converges to a critical point of $f$ and obtain new  convergence rates both for the values and the iterates. The analysis covers alternating versions of the forward-backward method with variable metric and relative errors. As an example, a nonsmooth and nonconvex version of the Levenberg-Marquardt algorithm is detailled.}

\vspace{0.5cm}


\vspace{0.5cm}


\vspace{0.5cm}

\noindent\hrulefill

\vspace{0.2cm}

\noindent {The second and third authors are partly supported by Conicyt Anillo Project ACT-1106, ECOS-Conicyt Project C13E03 and Millenium Nucleus ICM/FIC P10-024F. The third author is also partly supported by FONDECYT Grant 1140829 and Basal Project CMM Universidad de Chile.}

\vspace{0cm}

\noindent\hrulefill

\vspace{0.5cm}

\noindent {P. Frankel  \& G. Garrigos\\
Institut de Math\'ematiques et Mod\'elisation de Montpellier, UMR 5149 CNRS.\\
Universit\'e Montpellier 2, Place Eug\`ene Bataillon, 34095 Montpellier cedex 5, France.\\
\textit{Email:} p.frankel30@orange.fr, guillaume.garrigos@gmail.com}

\vspace{0.5cm}

\noindent {G. Garrigos \& J. Peypouquet\\
Departamento de Matem\'atica \& AM2V.\\
Universidad T\'ecnica Federico Santa Mar\'\i a, Avenida  Espa\~na 1680, Valpara\'\i so, Chile.\\
\textit{Email:} guillaume.garrigos@gmail.com, juan.peypouquet@usm.cl}

\vspace{0.5cm}

\noindent {\tiny Submitted: 11 November 2013.}

\end{abstract}

\maketitle
\thispagestyle{empty}

\newpage
\section{Introduction}

In this paper we present a class of numerical methods to find critical points for a class of nonsmooth and nonconvex functions defined on a Hilbert space. Our analysis relies on the Kurdyka-\L ojasiewicz (K\L ) inequality, initially formulated by \L ojasiewicz for analytic functions in finite dimension \cite{Loja63}, and later extended to nonsmooth functions in more general spaces \cite{Kur98,Loj93,KurPar,BolDanLew1}. Gradient-like systems governed by potentials satisfying this K\L \ inequality enjoy good asymptotic properties: under a compactness assumption, the corresponding trajectories have finite length and converge strongly to equilibria or critical points. These ideas were used in \cite{Sim83} to study nonlinear first-order evolution equations (see also \cite{HuaTak,ChiJen}). Second-order systems were considered in \cite{HarJen,Har} and a Schr\"odinger equation in \cite{BauSal}.

The convergence analysis of algorithms in this context is more recent. See \cite{AbsMahAnd} for gradient-related methods, \cite{attbol,BolDanLeyMaz10,MerPie10} for the proximal point algorithm and \cite{Noll} for a nonsmooth subgradient-oriented descent method. The celebrated Forward-Backward algorithm, a splitting method exploiting the nonsmooth/smooth structure of the objective function, has been studied in \cite{AttBolSva}, and extended in \cite{ChoPesRep13} to take in account a variable metric. Another splitting approach comes from Gauss-Seidel-like methods, which apply to functions with separated variables, and consist in doing a descent method relatively to each (block of) variables alternatively. See \cite{AttBolRedSou,XuYin} for a proximal alternating method, and \cite{AttBolSva} for a variable-metric version. Recent papers \cite{XuYin,BolSabTeb,ChoPesRep14} propose to combine these two splitting approaches in order to exploit both the smooth/nonsmooth character and the separated structure of the function. 
 
Most of the algorithms studied in the aforementioned papers share the same asymptotic behavior: under a compactness assumption, the sequences generated converge strongly to critical points, and the affine interpolations have finite length. This is not surprising since the algorithms described in \cite{attbol,MerPie10,BolDanLeyMaz10,AttBolSva,AttBolRedSou,BolSabTeb} together with the ones of \cite{ChoPesRep13,XuYin} (without extrapolation step) fall into the general convergence result for abstract descent methods of Attouch, Bolte and Svaiter \cite{AttBolSva}. Besides, these methods essentially share  the same hypotheses on the parameters with the abstract method of \cite{AttBolSva} : the step sizes (resp. the eigenvalues of the  matrices underlying the metric) are required to remain in a compact subinterval of the positive numbers. Moreover they have little flexibility regarding the presence of computational errors. To our knowledge vanishing step sizes (resp. unbounded eigenvalues) or sufficiently general errors have never been treated in the K{\L} context.

Another interesting aspect is that the convergence rate of several of these methods are essentially the same, and depend on the K{\L} inequality rather than the nature of the algorithm. Therefore, it seems reasonable to consider the existence of an abstract convergence rate result for general descent methods.

We present now the structure of the paper and underline its main contributions: in Section \ref{S:Preliminaries} we recall some definitions, well-known facts, and set the notation. Section \ref{S:Abstract} contains the main {\em theoretical} results of the paper. More precisely, in Subsection \ref{SS:capt_conv_lgth}, we present an abstract inexact descent method, which is inspired by \cite{AttBolSva} but extending their setting in order to account for additive computational errors and more versatility in the choice of the parameters. The strong convergence of the iterates with a {\em finite-length} condition, and a {\em capture} property are proved under certain hypotheses. Since the proofs are very close to those of \cite{AttBolSva}, most arguments are given in Appendix \ref{A:1}. Then, in Subsection \ref{SS:rates} we prove new and interesting general convergence rates. They are similar to the ones obtained in \cite{attbol,AttBolRedSou,XuYin,BolSabTeb,ChoPesRep14}. Surprisingly, an explicit form of the algorithm terminates in a finite number of iterations in several cases. A link with convergence rates for some continuous-time dynamical systems is also given. Sections \ref{S:specializations} and \ref{S:applications} contain the main {\em practical} contributions. In Section \ref{S:specializations}, we present a particular instance of the model, which provides further insight into a large class of known methods and present some innovative variants. More exactly, we revisit the \textit{Alternating Forward-Backward} methods, already considered in \cite{BolSabTeb,ChoPesRep14,LiPanChe}, but allowing inexact computation of the iterates and a dynamic choice of metric. This setting includes also the \textit{generalized Levenberg-Marquardt algorithm}, a Newton-like method adapted for nonconvex and nonsmooth functions. In Section \ref{S:applications}, we briefly describe an instance of this algorithm to produce a new method for the sparse and low-rank matrix decomposition.  Finally, some perspectives are discussed in Section \ref{SS:conclusion}.

%%%%%%%%%%%%%%%%%%%%%%%%%%%%%%%%%%%%%%%%%%%%%%%%%%%%%%%%%%%%%%%%%%%%%%%%%%%%%%%%%%%%%%%%%%%%%%%%%%%%%%%%%%%%%%%%
%%%%%%%%%%%%%%%%%%%%%%%%%%%%%%%%%%%%%%%%%%%%%%%%%%%%%%%%%%%%%%%%%%%%%%%%%%%%%%%%%%%%%%%%%%%%%%%%%%%%%%%%%%%%%%%%
%%%%%%%%%%%%%%%%%%%%%%%%%%%%%%%%%%%%%%%%%%%%%%%%%%%%%%%%%%%%%%%%%%%%%%%%%%%%%%%%%%%%%%%%%%%%%%%%%%%%%%%%%%%%%%%%
%%%%%%%%%%%%%%%%%%%%%%%%%%%%%%%%%%%%%%%%%%%%%%%%%%%%%%%%%%%%%%%%%%%%%%%%%%%%%%%%%%%%%%%%%%%%%%%%%%%%%%%%%%%%%%%%

\section{Preliminaries}\label{S:Preliminaries}

Throughout this paper $H$ is a real Hilbert space with norm $\Vert \cdot \Vert$ and scalar product $\langle \cdot , \cdot \rangle$.
 We write $x^{k}\longrightarrow x$, or $x^{k} \overset{w}{\longrightarrow} x$, if $x^{k}$ converges strongly or weakly to $x$, respectively, as $k\to+\infty$. The {\em domain} of $f:H\longrightarrow\Rcupinf$ is $\dom f=\{x:f(x)<+\infty\}$. A sequence $x^{k}$ {\em $f$-converges to $x$} (we write $x^{k} \overset{f}{\longrightarrow} x$) if $x^{k} \overset{}{\longrightarrow} x$ and $f(x^{k}) \overset{}{\longrightarrow} f(x)$. We say that a sequence is \textit{precompact} (resp. {\em $f$-precompact}) if it has at least one convergent (resp. $f$-convergent) subsequence.

\subsection{Subdifferential and critical points}

Let $f:H\rightarrow\Rcupinf$. The {\em Fréchet subdifferential} of $f$ at $x\in \dom f$ is the set $\partial_F f(x)$ of those elements $p \in H$ such that
$$\liminf\limits_{{y\rightarrow x,\ y \neq x}} \dfrac{f(y)-f(x) - \langle p,y-x \rangle}{\Vert y-x \Vert} \geq 0.$$
For $x \notin \dom f$, we set $\partial_F f(x)$:= $\emptyset$. The {\em (limiting Fr\'echet) subdifferential} of $f$ at $x\in \dom f$ is the set $\partial f(x)$ of elements $p \in H$ for which there exists  sequences $(x^{k})_{k\in\N}$ and $(p^k)_{k\in\N} \text{ in } H \text{ such that } x^{k} \overset{f}{\longrightarrow} x, \ p^k \overset{w}{\longrightarrow} p,  \text{ and } p^k \in \partial_F f(x^{k})$.
As before, $\partial f(x):=\emptyset$ for $x \notin \dom f$ and its domain is $\dom \partial f:= \{ x \in H: \partial f(x) \neq \emptyset \}$. 
This subdifferential satisfies the following chain rule : let $g_1,g_2$ and $h$ be extended real valued functions on $H_1$, $H_2$ and $H_1 \times H_2$ respectively. If $h$ is continuously differentiable in a neighbourhood of $(x_1,x_2)\in \dom g_1 \times \dom g_2$, the subdifferential of $f(x_1,x_2):=g_1(x_1) + g_2(x_2) + h(x_1,x_2)$ at $(x_1,x_2)$ is
\begin{equation}
\partial f(x_1,x_2)=\left(\frac{}{}\partial g_1(x_1) + \{\nabla_1 h(x_1,x_2)\} \,,\, \partial g_2(x_2) + \{\nabla_2 h(x_1,x_2)\}\right).
\end{equation}

We say that $x \in H$ is a \textit{critical point} if $0 \in \partial f(x)$. The {\em lazy slope} of $f$ at $x$ is $\Vert \partial f(x)\Vert_-:=\inf\limits_{p \in \partial f(x)} \|p\|$ if $x\in\dom\partial f$, and $+ \infty$ otherwise. This definition gives the following result:

\begin{lemma}\label{LemCritical}
If $x^{k} \overset{f}{\longrightarrow} x$ and $\liminf\limits_{n\to +\infty} \Vert \partial f(x^{k}) \Vert_- =0$, then $0 \in \partial f(x)$.
\end{lemma}

\subsection{The Kurdyka-\L ojasiewicz property}\label{SS:KL_property}

Let $\eta\in ]0,+\infty]$ and let $\varphi: [0,\eta[ \longrightarrow [0, +\infty [$ be a continuous concave function such that $\varphi (0) = 0$ and $\varphi$ is continuously differentiable on $]0,\eta[$ with $\varphi'(t)>0$ for all $t\in]0,\eta[$. A proper lower-semicontinuous function $f:H\to\Rcupinf$ has the {\em Kurdyka-\L ojasiewicz property} at a point $x^* \in \dom \partial f$ with {\em desingularizing function $\varphi$} if there exists $\delta>0$ such that the {\em Kurdyka-\L ojasiewicz inequality}
\begin{equation}\label{E:KL_ineq}
\varphi' (f(x)-f(x^*)) \Vert \partial f(x) \Vert_- \geq 1
\end{equation}
holds for all $x$ in the {\em strict local upper level} set
\begin{equation}\label{E:strict_luls}
\Gamma_\eta (x^*,\delta)=\{\,x\in H\ :\ \Vert x- x^* \Vert < \delta \text{ and }f(x^*)<f(x)<f(x^*) + \eta\,\}.
\end{equation}
A proper lower-semicontinuous function having the Kurdyka-\L ojasiewicz property at each point of $\dom \partial f$ is a {\em K{\L} function}. When $f$ is continuously differentiable, \eqref{E:KL_ineq} becomes $\Vert \nabla (\varphi \circ f) \Vert \geq 1$. This means that the more $f$ is flat around its critical points, the more $\varphi$ has to be steep around $0$, whence the term ``desingularizing". The K{\L} property reveals the possibility to reparameterize the values of $f$ in order to avoid flatness around the critical points. We shall see in Subsection \ref{SS:rates} that the growth of $\varphi$ has a direct impact on the convergence rate of optimization algorithms.

Semi-algebraic and bounded sub-analytic functions in finite dimension satisfy a K{\L} inequality (\cite{Loj93,KurPar,BolDanLew1}), as well as some, but not all, convex functions (see \cite{BolDanLeyMaz10} for details and a counterexample). See \cite{BolDanLewShi07,Dri00,DriMil96}, and the references therein, for more information in the general context of {\em o-minimal} functions. See \cite{HarJen10,Chi06} for characterizations in infinite-dimensional Hilbert spaces.

\subsection{Proximal operator in a given metric}\label{SS:prox_metric}

Let $\Sym (H)$ denote the space of bounded, uniformly elliptic and self-adjoint operators on $H$. Each $A\in \Sym(H)$ induces a metric on $H$ by the inner product $\langle x,y \rangle_A:= \langle Ax,y\rangle$, and the norm $\Vert x \Vert_{A}:= \sqrt{\langle x,x\rangle_A}$. We also set $\alpha(A)$ as the infimum of the spectral values of $A$, satisfying $\Vert x \Vert_A^2 \geq \alpha(A) \Vert x \Vert^2$ for all  $x\in H$.
Let $f : H \rightarrow \R \cup \{+\infty \}$, the \textit{proximal operator of $f$ in the metric induced by $A$} is the set-valued mapping $\prox_f^A : H \rightrightarrows H$, defined as
\begin{equation}\label{D:prox}
\prox_f^A (x) := \underset{y \in H}{\argmin}  \left\{ f(y) + \frac{1}{2}\Vert y-x \Vert^2_A \right\}.
\end{equation}
Observe that $\prox_f^A (x)\neq\emptyset$ if $f$ is weakly lower-semicontinuous and bounded from below (see \cite[Theorem 3.2.5]{Att_But_Mic}), which holds in many relevant applications. If $f$ is the indicator function of a set, then $\prox_f^A (x)$ is the {\em nearest point mapping} relatively to the metric induced by $A$.

%%%%%%%%%%%%%%%%%%%%%%%%%%%%%%%%%%%%%%%%%%%%%%%%%%%%%%%%%%%%%%%%%%%%%%%%%%%%%%%%%%%%%%%%%%%%%%%%%%%%%%%%%%%%%%%%
%%%%%%%%%%%%%%%%%%%%%%%%%%%%%%%%%%%%%%%%%%%%%%%%%%%%%%%%%%%%%%%%%%%%%%%%%%%%%%%%%%%%%%%%%%%%%%%%%%%%%%%%%%%%%%%%
%%%%%%%%%%%%%%%%%%%%%%%%%%%%%%%%%%%%%%%%%%%%%%%%%%%%%%%%%%%%%%%%%%%%%%%%%%%%%%%%%%%%%%%%%%%%%%%%%%%%%%%%%%%%%%%%
%%%%%%%%%%%%%%%%%%%%%%%%%%%%%%%%%%%%%%%%%%%%%%%%%%%%%%%%%%%%%%%%%%%%%%%%%%%%%%%%%%%%%%%%%%%%%%%%%%%%%%%%%%%%%%%%

\section{Convergence of an abstract inexact descent method}\label{S:Abstract}

Throughout this section, $f:H\rightarrow \R\cup\{+\infty\}$ is a proper function that is lower-semicontinuous for the strong topology. We shall adopt the notation given in Subsection \ref{SS:KL_property} concerning the K{\L} property, whenever it is invoked. We consider a sequence $(x^k)_{k\in\N}$, computed by means of an abstract algorithm satisfying the following hypotheses:\\

\noindent\underline{$\hone$ (\textit{Sufficient decrease}):}  For each $k\in \mathbb N$, for some $a_k >0$, $$f(x^{k+1})+a_k\norm{x^{k+1}-x^k}^2\leq f(x^k).$$

\noindent\underline{$\htwo$ (\textit{Relative error}):} For each $k\in\mathbb N$, for some $b_{k+1} >0$ and $\eps_{k+1} \geq 0$, $$b_{k+1} \Vert \partial f(x^{k+1}) \Vert_- \leq \norm{x^{k+1}-x^k}+\eps_{k+1}. $$

\noindent\underline{$\hthree$ (\textit{Parameters}):} The sequences $(a_k)_{k\in\N}$, $(b_k)_{k\in\N}$ and $(\eps_k)_{k\in\N}$ satisfy:
\begin{enumerate}
	\item[(i)] $a_k \geq \underline{a}>0$ for all $k \geq 0$.
   \item[(ii)] $\left(b_k\right)_{k\in\N}\notin l^1$;
   \item[(iii)] $\sup_{k\in\N^*}\frac{1}{a_kb_k}<+\infty$;
   \item[(iv)] $\left(\epsilon_k\right)_{k\in\N}\in l^1$.\\
\end{enumerate}

In Section \ref{S:specializations}, we complement this axiomatic description of descent methods by providing a large class of implementable algorithms that produce sequences verifying hypotheses $\hone$, $\htwo$ and $\hthree$. A simple example is:

\begin{example}
{ \em	If $f$ is differentiable, a gradient-related method (see \cite{Ber}) is an algorithms where each iteration has the form $x^{k+1}=x^k+\lambda_kd^k$, where $\lambda_k >0$ and $d^k$ agrees with the steepest descent direction $-\nabla f(x^k)$ in the sense that $\langle d^k,\nabla f(x^k)\rangle + C\norm{d^k}^{ 2}\le 0$ and $\norm{\nabla f(x^k)+d^k}\le C\norm{d^k}+e_k$, with $C>0$ and $\lim_{k\to\infty}e_k=0$. If $\nabla f$ is Lipschitz-continuous, it is easy to find conditions on the sequence $(\lambda_k)$ to verify hypotheses $\hone$, $\htwo$ and $\hthree$.}
\end{example}

\subsection{Capture, convergence, and finite length of the trajectories}\label{SS:capt_conv_lgth}

Sequences generated by the procedure described above converge strongly to critical points of $f$ and the piecewise linear curve obtained by interpolation has finite length. More precisely, we have:

\begin{theorem}\label{T:Tmain}
Let $f:H\rightarrow \R\cup\{+\infty\}$ be a K{\L} function and let $\hone$, $\htwo$ and $\hthree$ hold. If the sequence $(x^k)_{k\in\N}$ is $f$-precompact, then it $f$-converges to a critical point of $f$ and
$\sum_{k=0}^{+\infty} \norm{x^{k+1}-x^k}< + \infty$.
\end{theorem}

It is possible in Theorem \ref{T:Tmain} to drop the $f$-precompactness assumption and obtain a {\em capture result}, near a global minimum of $f$. To simplify the notation, for $x^*\in H$, $\eta\in]0,+\infty]$ and $\delta>0$, define the {\em relaxed local upper level set} by
\begin{equation}\label{E:relaxed_luls}
\underline{\Gamma}_\eta(x^*,\delta)= \{\,x\in H\ :\ \Vert x- x^* \Vert < \delta \text{ and }f(x^*)\leq f(x)<f(x^*) + \eta\,\}.
\end{equation}
We have the following:
\begin{theorem}\label{T:localCV}
Let $f:H \longrightarrow \Rcupinf$ have the K{\L} property in a global minimum $x^*$ of $f$. Let $(x^k)_{k\in\mathbb{N}}$ be a sequence satisfying $\hone$, $\htwo$ and $\hthree$ with $\epsilon_k \equiv 0$. Then, there exist $\gamma>0$ and $\eta>0$ such that if $x^0 \in \underline{\Gamma}_\eta(x^*,\gamma)$, then $(x^k)_{k\in\mathbb{N}}$ $f$-converges to a global minimum $\overline x$ of $f$, with $\sum_{k=0}^{+\infty} \norm{x^{k+1}-x^k}< + \infty$.
\end{theorem}

\noindent As mentioned in \cite{AttBolSva}, Theorem \ref{T:localCV} admits a more general formulation, for instance, if $x^*$ is a local minimum of $f$ where a growth assumption is locally satisfied (see \cite[Remark 2.11]{AttBolSva}).

The proofs of Theorems \ref{T:Tmain} and \ref{T:localCV} follow the arguments in \cite[Subsection 2.3]{AttBolSva}, adapted to the presence of errors and the variability of the parameters. They are given in Appendix \ref{A:1} for the reader's convenience.

%%%%%%%%%%%%%%%%%%%%%%%%%%%%%%%%%%%%%%%%%%%%%%%%%%%%%%%%%%%%%%%%%%%%%%%%%%%%%%%%%%%%%%%%%%%%%%%%%%%%%%%%%%%%%%%%

\subsection{Rates of Convergence}\label{SS:rates}

We assume that $\hone$, $\htwo$ and $\hthree$ hold, and for simplicity and precision, we restrict ourselves to the case where $\eps_k\equiv 0$. Suppose that $x^k$ $f$-converges to a point $x^*$ where $f$ has the K{\L} property.  We study three types of convergence rate results, depending on the nature of the desingularizing function $\varphi$:

\begin{enumerate}[label=\arabic*. ,leftmargin=0.5cm]
    \item Theorem \ref{T:CVvalues} establishes the relationship between the distance to the limit $\|x^k-x^*\|$ and the gap $f(x^k)-f(x^*)$, for a generic desingularizing function. It is similar to the result in \cite{BolDanLeyMaz10} for the proximal method in the convex case.
    \item Theorem \ref{T:explicit_rates} gives explicit convergence rates in terms of the parameters $-$ both for the distance and the gap $-$ when the desingularizing function is of the form $\varphi(t)=\frac{C}{\theta}t^\theta$ with $C>0$ and $\theta \in ]0,1]$. Several results obtained in the literature for various methods are recovered.
    \item Finally, Theorem \ref{T:CVexplicit} provides convergence rates when $\htwo$ is replaced by a slightly different hypothesis that holds for certain explicit schemes, namely gradient-related methods. This result is valid for a generic desingularizing function $\varphi$. However, when $\varphi$ is of the form $\varphi(t)=\frac{C}{\theta}t^\theta$ ($C>0$, $\theta \in ]0,1]$) the prediction is considerably better than the one provided by Theorem \ref{T:explicit_rates}. 
\end{enumerate}

\subsubsection{Distance to the limit in terms of the gap}

\begin{theorem}\label{T:CVvalues}
Set $\tilde{\varphi}(t):=\max\{\varphi(t),\sqrt{t}\}$. Then $\Vert x^* - x^k \Vert =O\left(\tilde{\varphi}(f(x^{k-1})-f(x^*))\right).$
\end{theorem}

\begin{proof}
By assumption, $x^{k} \overset{f}{\longrightarrow} x^*$ and $f$ satisfies the K{\L} inequality on some $\Gamma_\eta(x^* ,\delta)$. Let $r_k:=f(x^k)-f(x^*) \geq 0$. We may suppose that $r_k>0$ for all $k\in\N$ because otherwise the algorithm terminates in a finite number of steps. For $K$ large enough, we have $x^{k} \in \Gamma_\eta (x^* , \delta)$ for all $k\geq K$. Lemma \ref{L:L1}, gives
$$2 \Vert x^{k+1} - x^k \Vert \leq \Vert x^k - x^{k-1} \Vert + M[\varphi(r_k) - \varphi(r_{k+1})]$$
for all $k\geq K$. Summing this inequality for $k=K,\dots,N$, we obtain
$${\sum\limits_{k=K}^{N} \Vert x^{k+1}-x^k \Vert}\leq \Vert x^K - x^{K-1} \Vert + M \varphi(r_{K}).$$
Using the triangle inequality and passing to the limit, we get
$$\Vert x^* - x^{K} \Vert  \leq  \sum\limits_{k=K}^\infty \Vert x^{k+1} - x^k \Vert \leq \Vert x^{K}-x^{K-1} \Vert + M \varphi (r_{K}) \leq \frac{\sqrt{f(x^{K-1})-f(x^{K})}}{\sqrt{a_K}} + M \varphi(r_{K})$$
by $\hone$. Then, using ${\mathbf H}_0$, along with the fact that $f(x^{K}) \geq f(x^*)$ and that $(r_k)$ is decreasing, we deduce that $\Vert x^* - x^{K} \Vert\leq  \frac{1}{\sqrt{\underline{a}}} \sqrt{r_{K-1}} + M \varphi(r_{K-1})$,
which finally gives $\Vert x^* - x^{K} \Vert\le\max\left\{\frac{1}{\sqrt{\underline{a}}},M\right\}\tilde{\varphi}(r_{K-1})$.
\bx\end{proof}

\subsubsection{Explicit rates when $\varphi(t)=\frac{C}{\theta}t^\theta$ with $C>0$ and $\theta \in ]0,1]$}

Theorem \ref{T:explicit_rates} below is qualitatively analogous to the results in \cite{attbol,MerPie10,AttBolRedSou,XuYin,BolSabTeb,ChoPesRep14} : convergence in a finite number of steps if $\theta=1$, exponential convergence if $\theta\in[\frac{1}{2},1[$ and polynomial convergence if $\theta\in]0,\frac{1}{2}[$. In the general convex case, finite-time termination of the proximal point algorithm was already proved in \cite{Rock} and \cite{Fer} (see also \cite{Pey_2009}).

\addtocounter{footnote}{+1}

\begin{theorem}\label{T:explicit_rates}
Assume $\varphi(t)=\frac{C}{\theta}t^\theta$ for some $C>0$, $\theta \in ]0,1]$.
\begin{itemize}
		\item [i)] If $\theta=1$ and $\inf\limits_{k\in\N} a_k b_{k+1}^2 >0$,\footnotemark[\value{footnote}] then $x^k$ converges in finite time.
		\item [ii)] If $\theta \in [\frac{1}{2},1[$, $\sup\limits_{k\in\N} b_k<+\infty$ and $\inf\limits_{k\in\N}a_k b_{k+1}>0$,\footnotemark[\value{footnote}] there exist $c>0$ and $k_0 \in\N$ such that:
		\begin{enumerate}[label=\arabic*. ,leftmargin=0.5cm]
			\item $f(x^k)-f(x^*)=O\left(\exp\left(-c {\sum\limits_{n=k_0}^{k-1} b_{n+1}}\right)\right)$, and
			\item $\Vert x^* - x^k\Vert=O\left(\exp\left(-\dfrac{c}{2} {\sum\limits_{n=k_0}^{k-2} b_{n+1}}\right)\right)$.
		\end{enumerate}
		\item [iii)] If $\theta \in ]0,\frac{1}{2}[$, $\sup\limits_{k\in\N} b_k<+\infty$ and $\inf\limits_{k\in\N}a_k b_{k+1}>0$,\footnotemark[\value{footnote}] there is $k_0 \in\N$ such that:
		\begin{enumerate}[label=\arabic*. ,leftmargin=0.5cm]
			\item $f(x^k)-f(x^*)=O\left( \left({\sum\limits_{n=k_0}^{k-1} b_{n+1}}\right)^{\frac{-1}{1- 2\theta}}\right)$, and
			\item $\Vert x^* - x^k\Vert =O\left( \left({\sum\limits_{n=k_0}^{k-2} b_{n+1}}\right)^{\frac{-\theta}{1- 2\theta }}   \right)$.
		\end{enumerate}
\end{itemize}\footnotetext[\value{footnote}]{A simple sufficient $-$ yet not necessary $-$ condition for $\inf_{k\in\N} a_k b_{k+1}^2 >0$ and $\inf_{k\in\N} a_k b_{k+1}>0$ is that $\inf_{k\in\N}b_k>0$.}
\end{theorem}

\begin{proof}
We can suppose that $r_k >0$ for all $k\in\N$, because otherwise the algorithm terminates in a finite number of steps. Since $x^k$ converges to $x^*$, there exists $k_0 \in \N$ such that for all $k\geq k_0$ we have $x^k \in \Gamma_\eta(x^*,\delta)$ where the K{\L} inequality holds. Using successively  $\hone$, $\htwo$ and the K{\L} inequality we obtain
	\begin{eqnarray}\label{EQh1h2KL}
		\varphi'^2(r_{k+1}) (r_k - r_{k+1})  \geq   \varphi'^2(r_{k+1}) a_k b_{k+1}^2 \Vert  \partial f(x^{k+1}) \Vert_{-}^2 \geq  a_k b_{k+1}^2
	\end{eqnarray}
for each $k\geq k_0$. Let us now consider different cases for $\theta$:

\noindent\underline{Case $\theta=1$:} If $r_k >0$ for all $k\in\N$, then $C^2 (r_k - r_{k+1}) \geq a_k b_{k+1}^2\ge \inf\limits_{k\in\N} a_k b_{k+1}^2>0$
for all $k\ge k_0$. Since $r_k$ converges, we must have $\inf\limits_{k\in\N} a_k b_{k+1}^2=0$, which is a contradiction. Therefore, there exists some $k \in\N$ such that $r_k=0$, and the algorithm terminates in a finite number of steps.
		
\noindent\underline{Case $\theta \in ]0,1[$:} Write $\bar b:=\sup\limits_{k\in\N} b_k$, $m:=\inf\limits_{k\in\N}a_k b_{k+1}$ and $c=\frac{m}{C^2(1+\bar b)}$ and, for each $k\in\N$, $\beta_k:=\frac{b_km}{C^2}$. For each $k\ge k_0$, inequality (\ref{EQh1h2KL}) gives
\begin{equation}\label{E:explicit_rates}
(r_k - r_{k+1}) \geq \frac{a_k b_{k+1}^2r_{k+1}^{2-2\theta}}{C^2} \geq \beta_{k+1} r_{k+1}^{2-2\theta}.
\end{equation}

\noindent\underline{Subcase $\theta \in [\frac{1}{2},1[$:} Since $r_k\to 0$ and $0<2-2\theta\le 1$, we may assume, by enlarging $k_0$ if necessary, that $r_{k+1}^{2-2\theta}\ge r_{k+1}$ for all $k\geq k_0$. Inequality \eqref{E:explicit_rates} implies $(r_k - r_{k+1}) \geq \beta_{k+1} r_{k+1}$ or, equivalently, $r_{k+1} \leq r_k \left( \dfrac{1}{1+\beta_{k+1}} \right)$ for all $k\geq k_0$. By induction, we obtain
$$r_{k+1} \leq r_{k_0} \left( \prod\limits_{n=k_0}^k \dfrac{1}{1+\beta_{n+1}} \right) = r_{k_0}\exp\left(\sum\limits_{n=k_0}^{k} \ln \left(\dfrac{1}{1+\beta_{n+1}} \right)\right)$$
for all $k\geq k_0$. But $\ln \left(\dfrac{1}{1+\beta_{n+1}} \right) \leq \dfrac{-\beta_{n+1}}{1+\beta_{n+1}}  \leq \dfrac{-1}{1 + \bar b}\beta_{n+1}$, and so
$$r_{k+1} \leq r_{k_0}\exp\left\{\sum\limits_{n=k_0}^{k}\left(\dfrac{-1}{1 + \bar b}\beta_{n+1} \right)\right\}=r_{k_0}\exp\left(-c\sum\limits_{n=k_0}^{k} b_{n+1}\right).$$
The second part follows from Theorem \ref{T:CVvalues}.
		
\noindent\underline{Subcase $\theta \in ]0,\frac{1}{2}[$:} Recall from inequality \eqref{E:explicit_rates} that $r_{k+1}^{2\theta -2} (r_k - r_{k+1}) \geq \beta_{k+1} $. Set $\phi(t):=\frac{C}{1-2\theta} t^{2\theta -1}$. Then $\phi'(t) = -Ct^{2\theta -2}$, and
$$\phi(r_{k+1}) - \phi(r_k) = \int\limits_{r_k}^{r_{k+1}}\phi'(t)\,dt= C \displaystyle\int\limits_{r_{k+1}}^{r_k} t^{2\theta -2 }\,dt \geq C(r_k - r_{k+1}) r_k^{2\theta-2}.$$
On the one hand, if we suppose that $r_{k+1}^{2\theta -2} \leq 2 r_k^{2\theta -2}$, then
$$\phi(r_{k+1}) - \phi(r_k) \geq \frac{C}{2}(r_k - r_{k+1}) r_{k+1}^{2\theta-2} \geq \frac{C}{2} \beta_{k+1}.$$
On the other hand, suppose that $r_{k+1}^{2\theta -2}>2 r_k^{2\theta -2}$. Since $2\theta-2<2\theta-1<0$, we have $\frac{2\theta-1}{2\theta-2}>0$. Thus $r_{k+1}^{2\theta -1} > q r_k^{2\theta -1}$, where $q:=2^\frac{2\theta -1}{2\theta -2} >1$. Therefore,
$$\phi(r_{k+1})-\phi(r_k) = \frac{C}{1-2\theta}(r_{k+1}^{2\theta -1} - r_k^{2\theta -1})>\frac{C}{1-2\theta} (q-1) r_k^{2\theta -1} \geq C',$$
with $C':=\frac{C}{1-2\theta} (q-1) r_{k_0}^{2\theta-1}>0$. Since $\beta_{k+1}\le \frac{\bar bm}{C^2}$, we can write
$$\phi(r_{k+1})-\phi(r_k)\ge \frac{C'C^2}{\bar bm}\beta_{k+1}.$$
Setting $c:=\min\{\frac{C}{2},\frac{C'C^2}{\bar bm}\}>0$ we can write $\phi(r_{k+1})-\phi(r_k)\ge c\beta_{k+1}$
for all $k\ge k_0$. This implies
$$\phi(r_{k+1})\geq \phi(r_{k+1}) -\phi(r_{k_0}) = {\sum\limits_{n=k_0}^{k} \phi(r_{n+1})-\phi(r_n) } \geq c \sum\limits_{n=k_0}^{k} \beta_{n+1},$$
which is precisely
$r_{k+1}\le D\left(\sum\limits_{n=k_0}^{k} b_{n+1}\right)^{\frac{-1}{1-2\theta}}$
with $D=\left(\frac{cm(1-2\theta)}{C^3}\right)^{\frac{-1}{1-2\theta}}$. As before, Theorem \ref{T:CVvalues} gives the second part.
\bx\end{proof}

\subsubsection{Sharper results for gradient-related methods}

Convergence rates for the continuous-time gradient system
\begin{equation}\label{E:steepest}
-\dot x(t)=\nabla f(x(t)),
\end{equation}
where $f$ is some integral functional,
are given in \cite{ChiFio06}. For any $\varphi$, \cite[Theorem 2.7]{ChiFio06} states that
	\begin{enumerate}
		\item $f(x^k)-f(x^*)= O\left(\Phi^{-1}(t-\hat t)\right)$, and
		\item $\Vert x^*-x^k\Vert_{L^2(\Omega)} = O\left(\varphi\circ \Phi^{-1}(t-\hat t)\right)$,
	\end{enumerate}	
where $\Phi$ is any primitive of $-(\varphi')^2$. If the desingularizing function $\varphi$ has the form $\varphi(t)=\frac{C t^\theta}{\theta}$, we recover (see \cite[Remark 2.8]{ChiFio06}) convergence in finite time if $\theta\in ]\frac{1}{2},1]$,
exponential convergence if $\theta = \frac{1}{2}$, and
polynomial convergence if $\theta \in ]0,\frac{1}{2}[$. The same conclusion was established in \cite[Theorem 4.7]{BolDanLew1} for a nonsmooth version of \eqref{E:steepest} when $f$ is any subanalytic function in $\R^N$. This prediction is better than the one given by Theorem \ref{T:explicit_rates} above, as well as the results in \cite{attbol,MerPie10,AttBolRedSou,XuYin,BolSabTeb,ChoPesRep14}  since it guarantees convergence in finite time for $\theta>\frac{1}{2}$. We shall prove that for certain algorithms including gradient-related methods, this better estimation remains true. To this end, consider the following variant of hypothesis $\htwo$:
		
\noindent\underline{$\htwo'$ (\textit{Relative error}):} For each $k\in\mathbb N$, $b_{k+1}\norm{\partial f(x^k)}_- \leq \norm{x^{k+1}-x^k}$.

\begin{theorem}\label{T:CVexplicit}
Suppose condition $\htwo '$ is satisfied instead of $\htwo$ and assume $m:=\inf\limits_{k\in\N}a_k b_{k+1}>0$. Let $\Phi:]0,\eta[\to\R$ be any primitive of $-(\varphi')^2$.
\begin{itemize}
	\item [i)] If $\lim\limits_{t\to 0} \Phi(t) \in \R$, then the algorithm converges in a finite number of steps.
	\item [ii)] If $\lim\limits_{t\to 0} \Phi(t) = +\infty$, then there exists $k_0\in\N$ such that:
	\begin{enumerate}[label=\arabic*. ,leftmargin=0.5cm]
		\item $f(x^k) - f(x^*) = O\left( \Phi^{-1}\left(m\sum\limits_{n=k_0}^{k-1} b_{n+1}\right)\right)$, and
		\item $\Vert x^* - x^k \Vert = O\left(\varphi\circ \Phi^{-1}\left(m\sum\limits_{n=k_0}^{k-1} b_{n+1}\right)\right)$.
	\end{enumerate}	
\end{itemize}
\end{theorem}

\begin{proof}
The following proof is inspired by the one of \cite{ChiFio06} in the continuous case. First, if $r_k>0$ for all $ k\in \N$, we claim that there is $k_0\in\N$ such that
\begin{equation}\label{EQ:FactExplicite}
\Phi(r_{k+1})\geq \Phi(r_{k_0})+m\sum\limits_{n=k_0}^{k}b_{n+1}.
\end{equation}
To see this, let $k_0$ be large enough to have $x^k \in \Gamma_\eta(x^*,\delta)$ where the K{\L} inequality holds for all $k\geq k_0$. We apply successively $\hone$, $\htwo'$, the K{\L} inequality and $\hthree$ to obtain
\begin{equation*}
\varphi'(r_{k})^2(r_k - r_{k+1})  \geq   \varphi'(r_{k})^2a_kb_{k+1}^2 \Vert \partial f(x^{k}) \Vert_{-}^2 \geq  a_k b_{k+1}^2 \geq b_{k+1}m.
\end{equation*}
Let $\Phi$ be a primitive of $-(\varphi')^2$. Then
$$\Phi(r_{k+1}) - \Phi(r_{k}) = \displaystyle \int _{r_{k+1}}^{r_{k}} \varphi'(t)^2\, dt \geq (r_{k} - r_{k+1} ) \varphi'(r_{k})^2 \geq b_{k+1}m$$
because $\varphi'$ is decreasing. Therefore,
$$\Phi(r_{k+1})- \Phi(r_{k_0}) = {\sum\limits_{n=k_0}^{k} \Phi(r_{n+1})-\Phi(r_n)} \geq m\sum\limits_{n=k_0}^{k}b_{n+1}$$
as claimed. Now let us analyze the two cases:

For i), if $ r_k>0 $ for all $ k\in \N$, then \eqref{EQ:FactExplicite} implies $\lim\limits_{k\to +\infty} \Phi(r_{k+1}) = +\infty$ which contradicts the fact that $\lim\limits_{t\to 0} \Phi(t) \in \R$. Hence, $r_k=0$ for some $k\in\N$.

For ii), we may suppose that $ r_k>0 $ for all $ k\in \N$ (otherwise the algorithm stops in a finite number of steps) and so \eqref{EQ:FactExplicite} holds for all $k\in\N$. Since $\lim\limits_{k \to +\infty} \Phi(r_k) = +\infty$, we can take $k_0$ large enough to have $\Phi(r_{k_0})>0$. Whence $\Phi(r_{k+1})\geq m\sum\limits_{n=k_0}^{k} b_{n+1}$. Since $(b_n) \notin\ell^1$, for all sufficiently large $k$, $m\sum\limits_{n=k_0}^{k} b_{n+1}$ is in the domain of $\Phi^{-1}$ and we obtain the first estimation, namely:
\begin{equation}\label{E:explicit2}
r_{k+1} \leq \Phi^{-1}\left(m\sum\limits_{n=k_0}^{k} b_{n+1}\right).
\end{equation}
For the second one, since $\varphi$ is concave and differentiable, we have
$$
\varphi(r_{k}) - \varphi(r_{k+1}) \geq \varphi' (r_{k}) (r_{k} - r_{k+1}) \geq \varphi' (r_{k}) a_n \Vert x^{k+1} -x^{k} \Vert^2,
$$
by $\hone$. The K{\L} property and $\htwo'$ then give
$$\varphi(r_{k}) - \varphi(r_{k+1}) \ge m\Vert x^{k+1} - x^{k} \Vert,$$
which in turn yields
$$\Vert x^* - x^{k} \Vert \leq \frac{1}{m}\sum_{n=k}^\infty\left[\varphi(r_{n})-\varphi(r_{n+1})\right]\le \frac{1}{m}\varphi(r_{k}).$$
We conclude by using \eqref{E:explicit2}.
\bx\end{proof}

%%%%%%%%%%%%%%%%%%%%%%%%%%%%%%%%%%%%%%%%%%%%%%%%%%%%%%%%%%%%%%%%%%%%%%%%%%%%%%%%%%%%%%%%%%%%%%%%%%%%%%%%%%%%%%%%

\section{Descent methods with errors and variable metric}\label{S:specializations}

As stressed in \cite{AttBolSva}, the abstract scheme developed in Section \ref{S:Abstract} covers, among others, the gradient-related methods (a wide variety of schemes based on the gradient method sketched in \cite{Cauchy}), the proximal algorithm (introduced in \cite{Martinet} and further developed in \cite{BreLio,Rock}), and the forward-backward algorithm (a combination of the preceding, see \cite{LM,Passty}). This last one is a splitting method, used to solve structured optimization problems with the following form
\begin{equation}
\underset{x\in H}{\text{\rm minimize}} \ f(x)=g(x)+h(x),
\end{equation}
where $g$ is a nonsmooth proper l.s.c function and $h$ is differentiable with a $L$ Lipschitz gradient. It has been studied in the  nonsmooth nonconvex setting in \cite{AttBolSva} and the algorithm was stated as follows: start with $x^0 \in H$, consider $(\lambda_k) \subset [ \underline{\lambda},\bar \lambda ] \text{ with } 0<\underline{\lambda} \leq \bar \lambda < \frac{1}{L}$ and $\forall k\in \N$
\begin{equation}\label{D:FBclassic}
x^{k+1} \in \prox_{\lambda_k g} \left( x^k - \lambda_k \nabla h(x^k)\right).
\end{equation}
It satisfies $\hone$, $\htwo$ and $\hthree$ (see \cite[Theorem 5.1]{AttBolSva}) and falls into the setting of Theorem \ref{T:Tmain}. We shall extend this class of algorithms in different directions:

\begin{itemize}[leftmargin=1cm]
	\item Alternative choice of metric for the ambient space, which may vary at each step (see \cite{AlvBolBra,AlvLopRam} and the references therein). Considering  metrics induced by a sequence $(A_k) \subset \Sym (H)$, the forward-backward method becomes
\begin{equation}\label{D:FBmetric}
	x^{k+1} \in \prox_g^{A_k} \left(x^k - A_k^{-1} \nabla h(x^k) \right)
\end{equation}
(recall Subsection \ref{SS:prox_metric}). Indeed, (\ref{D:FBmetric}) can be rewritten as
\begin{equation}\label{D:FBmetric2}
	x^{k+1} \in \underset{y \in H}{\argmin} \ g(y) + h(x^k) + \langle y-x^k , \nabla h(x^k) \rangle + \frac{1}{2}\langle y-x^k , A_k (y-x^k) \rangle.
\end{equation}
At each step, an approximation of $f$, replacing its smooth part $h$ by a quadratic model, is minimized. See \cite{ChoPesRep13} for a similar algorithm called Variable Metric Forward-Backward, and \cite{Noll} for an approach considering more general models.
Note that when $A_k = \frac{1}{\lambda_k} id_H$ one recovers (\ref{D:FBclassic}). Allowing variable metric can improve convergence rates, help to implicitly deal with certain constraints, or compensate the effect of ill-conditioning. Rather than simply giving a convergence result for a general choice of $A_k$, we handle, in Subsection \ref{SS:newton}, a detailed method to select these operators, using second-order information.
	
	\item Effectively solve structured problems as
\begin{equation}
\underset{x_1\in H_1,  x_2 \in H_2}{\text{\rm minimize}} \ f(x_1,x_2)=g_1(x_1)+ g_2(x_2)+h(x_1,x_2),
\end{equation}
where $g_1, g_2$ are nonsmooth proper l.s.c functions and $h$ is differentiable with Lipschitz gradient. One approach is the regularized Gauss-Seidel method, which exploits the fact that the variables are separated in the nonsmooth part of $f$, as considered in \cite{AttBolSva,AttBolRedSou,XuYin}. It consists in minimizing alternatively a regularized version of $f$ with respect to each variable. In other words, it is an alternating proximal algorithm, of the form:
$$\begin{array}{|l}
\ x_1^{k+1} \in \prox_{f(\cdot, x_2^k)} \left(x_1^k\right) \\
\ x_2^{k+1} \in \prox_{f(x_1^{k+1},\cdot)} \left(x_2^k\right).
\end{array}$$
But this algorithm does not exploit the smooth nature of $h$. An alternative is to use an alternating minimization method which can deal with the nonsmooth character, while it benefits from the smooth features. An {\em Alternating Forward-Backward Method} considering variable metrics is presented below. A constant-metric version, namely the Proximal Alternating Linearized Minimization Algorithm, can be found in \cite{BolSabTeb}. A forthcoming paper \cite{ChoPesRep14} deals with the same algorithm, called {\em Block Coordinate Variable Metric Forward-Backward}, with a non-cyclic way of selecting the variables to minimize. Nevertheless, our setting differs from the aforementioned works in the following ways:
	\begin{itemize}[label=\textbullet,leftmargin=0.5cm]
	\item We allow more flexibility in the choice of parameters, accounting, in particular, for vanishing step sizes or unbounded eigenvalues for the metrics.
	\item We allow relative errors. Indeed, the computation of $\tilde{x}^k := x^k - A_k^{-1} \nabla h(x^k)$ and $x^{k+1} \in  \prox_g^{A_k} \left( \tilde{x}^k \right)$ often require solving some subroutines, which may produce $\tilde{x}^k$ and $x^{k+1}$ inexactly. To take these errors into account we introduce two sequences $(r^k)$, $(s^k)$ and consider
	\begin{equation}\label{D:FBerrors}
	x^{k+1}- s^{k+1}\in \prox_g^{A_k} \left(x^k - A_k^{-1} \nabla h(x^k) + r^k \right).
	\end{equation}
	Convergence of this method with errors is given in Theorem \ref{T:CVAFB}.
	\end{itemize} 
\end{itemize}

\subsection{The Alternating Forward-Backward (AFB) method}\label{SS:AFB}

Let $H_1,\dots,H_p$ be Hilbert spaces, each $H_i$ provided with its own inner product $\langle\cdot,\cdot\rangle_{H_i}$ and norm $\|\cdot\|_{H_i}$. If there is no ambiguity, we will just note $\Vert x_i \Vert$ instead of $\Vert x_i \Vert_{H_i}$. Set $H=\prod\limits_{i=1}^p H_i$ and endow it with the inner product $\langle \cdot , \cdot \rangle = {\sum\limits_{i=1}^{p} \langle \cdot , \cdot \rangle_{H_i}}$ and the associated norm $\Vert \cdot \Vert = \sqrt{\langle \cdot , \cdot \rangle}$. Consider the problem
\begin{equation}\label{E:structure}
\underset{x_i\in H_i}{\text{\rm minimize}} \ f(x_1,\dots,x_p)=h(x_1,\dots,x_p) + \sum\limits_{i=1}^{p} g_i(x_i),
\end{equation}
where $h: H\rightarrow \R$ is continuously differentiable and each $g_i: H_i\rightarrow {\R \cup \{+\infty\} }$ is a lower-semicontinuous function. Moreover we suppose that there is $L \geq 0$ such that for each $(x_1,...,x_p) \in H$ and $i\in \{1,...,p\}$, the application
\begin{equation}\label{D:PartialH}
x \in H_i  \mapsto h(x_1,...,x_{i-1},x,x_{i+1},...,x_p)
\end{equation}
has a $L$-Lipschitz continuous gradient. We shall present an algorithm that generates sequences converging to critical points of $f$. The sequences will be updated cyclically, meaning that given $(x_1^k,...,x_p^k)$, we start by updating the first variable $x_1^k$ into $x_1^{k+1}$, and then we consider $(x_i^{k+1},x_2^k,...,x_p^k)$ to update the second variable, and so on. In order to have concise and clear notations, throughout this section we shall denote:
\begin{equation}\label{D:VarIntermediaire}
X^k:=(x_1^k,...,x_p^k) \text{ and } X_i^k:=(x_1^{k+1},...,x_{i-1}^{k+1},x_i^k,...,x_p^k).
\end{equation}
Observe that $X_1^k=X^k$ and that we can write $X_{p+1}^k=X^{k+1}$.

Let us now present the Alternate Forward-Backward (AFB) algorithm. As said before, it consists in doing a forward-backward step relatively to each variable, taking in account  a possibly different metric. Then for all $i \in \{1,...,p\}$, consider a sequence $(A_{i,k}) \subset \Sym (H_i)$ which will model the metrics. Given a starting point $X^0 \in H$, the AFB algorithm generates a sequence $(X^k)$ by taking for all $k \in \N$ and $i \in \{1,...,p\}$
\begin{align}
\text{\textbf{(AFB)}} \ \ & \ \ x_i^{k+1} \in \prox_{g_i}^{A_{i,k}} \left( x_i^k - A_{i,k}^{-1} \nabla_i h (X_i^k)\right).
\end{align}

\noindent We shall consider some hypotheses on the operators $A_{i,k}$. Define $\alpha_k = \min\limits_{i=1..p}\alpha(A_{i,k})$ and $\beta_k:= \max\limits_{i=1..p} \VERT A_{i,k} \VERT$, which give bounds on the spectral values of $(A_{i,k})_{i=1..p}$. We make the following assumptions:

$$\begin{array}{ll}
\text{\textbf{(HP)}}  & \ 1. \ \text{ There exists } \underline{\alpha} > 0  \text{ such that } {\alpha}_k \geq \underline{\alpha} > L \\
 & \ 2. \ \frac{1}{\beta_k} \notin \ell^1 \ \hspace{1cm} 3. \ \sup\limits_{k\in\N} \frac{\beta_{k}}{{{\alpha}_{k+1}} } <+\infty.
\end{array}$$

\begin{remark}\label{R:param}
{\em Here $\mathbf{HP}_1$ is a bound on the spectral values by the Lipschitz constant of the gradient of $h$, in order to enforce the descent property of the sequence. For operators of the form $\frac{1}{\lambda_{i,k}} id_{H_i}$, we recover the classical bound $L \lambda_{i,k} \leq L \bar \lambda < 1$. In \cite{ChoPesRep14}, the authors prove that, with an additional convexity assumption on the $g_i$'s, and boundedness of the parameters, one can consider $L \lambda_{i,k} \leq L \bar \lambda < 2$. Item $\mathbf{HP}_2$ states that the spectral values may diverge, but not too fast. Finally, $\mathbf{HP}_3$ can be seen as an hypothesis on the variations of the extreme spectral values of the chosen operators. It clearly holds for instance if $\beta_k$ is bounded. It is also sufficient to assume that the condition numbers $$\kappa_i^k:= \frac{\VERT A_i^k \VERT}{\alpha( A_i^k)}$$ are bounded, with also $\min \left\{ \frac{\alpha_k}{\alpha_{k+1}} ,  \frac{\beta_k}{\beta_{k+1}} \right\}$ remaining  bounded.}
\end{remark}

\begin{remark}\label{R:Lipschitz}
{\em Even if $\nabla h$ is globally Lipschitz continuous, $L$ is not the Lipschitz constant of $\nabla h$ but a common Lipschitz constant for the functions defined in (\ref{D:PartialH}). As a consequence the partial gradients $\nabla_i h$ are $\sqrt{p}L$-Lipschitz continuous while $\nabla h$ is $pL$-Lipschitz. This allows us to have a better bound in $\mathbf{HP}_1$ which is of particular importance in the applications (see Section \ref{S:applications}). In \cite{BolSabTeb}, the authors give a more precise analysis: at each substep $X_i^k$ of the algorithm, they consider $L_{i,k}$ as the Lipschitz constant of the gradient of $x \in H_i  \mapsto h(x_1^{k+1},...,x_{i-1}^{k+1},x,x_{i+1}^k,...,x_p^k)$. Then they take step sizes equal to $\lambda_{i,k}=\frac{\epsilon_i}{L_{i,k}}$ where $\epsilon_i <1$ is a fixed non-negative constant. This approach can be related to the one in \cite{ChoPesRep13,ChoPesRep14}. However, they suppose {\em a priori} that the values $L_{i,k}$ remain bounded. It would be interesting to know if it is possible to combine the two approaches (a variable Lipschitz constant and vanishing step sizes).}
\end{remark}

\subsection{The AFB method with errors}

In order to allow for approximate computation of the descent direction or the proximal mapping, we go further by considering an inexact AFB method. We introduce the sequences $(r_i^k)$ and $(s_i^k)$ for $i\in\{1,...,p\}$ which correspond respectively to errors arising at the explicit and implicit steps relatively to the variable $x_i$. The  \textit{AFB method with Errors}  is computed from an initial  $(x_1^0,...,x_p^0)\in H$ by
$$\begin{array}{lrl}
\text{\textbf{(AFBE)}} \hspace{0.5cm} & y_i^{k+1} & \in \prox_{g_i}^{A_{i,k}}  \left( x_i^k - A_{i,k}^{-1} \nabla_i h(X_i^k) + r_i^k \right), \\
 & x_i^{k+1} & = y_i^{k+1} + s_i^{k+1}.
\end{array}$$

\noindent We do specific hypothesis on the errors  in view to guarantee the convergence of the method. Observe in particular that we do not assume a priori that the errors converge to zero:
$$\begin{array}{ll}
\text{\textbf{(HE)}}  &  \text{ There exists } \sigma \in [0,+\infty [, \rho \in ]0,1] \text{ with } \frac{\sigma +1}{\rho} < \underline{\alpha}L^{-1}  \text{ such that } \\
 & \ 1. \ \Vert S_i^k\Vert \leq \frac{\sigma}{2} \Vert y_i^{k+1} - y_i^k \Vert,   \text{ with } S_i^k \text{  defined from } (s_i^k) \text{ as in (\ref{D:VarIntermediaire}),} \\
 & \ 2. \ \Vert r_i^k \Vert \leq \frac{\sigma}{2} \Vert y_i^{k+1} - y_i^k \Vert + \mu_{k}, \text{ where } \mu_k \geq 0  \text{ with } \mu_k \in \ell^1, \\
 & \ 3. \ \langle r_i^k + s_i^k , y_i^{k+1} - y_i^k \rangle_{A_{i,k}}  \leq \frac{ 1-\rho}{2} \Vert y_i^{k+1} - y_i^k \Vert^2_{A_{i,k}}.
\end{array}$$

 This AFB algorithm (with errors) is related to the abstract descent method studied in Section \ref{S:Abstract}. This is stated in the next proposition, whose proof is left in Appendix \ref{Ap:2}.

\begin{proposition}\label{P:hi}
Any sequence $Y^k=(y_1^k,...,y_p^k)$ generated by the AFB algorithm with errors satisfies $\hone$, $\htwo$ and $\hthree$.
\end{proposition}

\noindent Given this result, one could directly apply Theorem \ref{T:Tmain} to obtain convergence of the sequence $(Y^k)$ to a critical point of $f$. But this result would suffer from some drawbacks. First, we are expecting that $(X^k)$ converges to a critical point, not $(Y^k)$. So we should make the assumption that the errors $S^k:= X^k - Y^k$ tend to zero. Moreover we would suppose that $(Y^k)$ is $f$-precompact, while we may only have an access to $(X^k)$. To handle this, we make the link between the asymptotic behaviour of $(Y^k)$ and $(X^k)$:

\begin{proposition}\label{P:LimitPoints}
For any sequence generated by the AFB method with errors:
\begin{enumerate}[label=\arabic*. ,leftmargin=0.5cm]
	\item If $(Y^k)$ has finite length, then so does $(X^k)$.
	\item If $(f(Y^k))$ is bounded from below then for all $i\in\{1..p\}$, $\Vert s_i^k \Vert$ and $\Vert r_i^k \Vert$ lie in $\ell^2$. In particular $(Y^k)$ and $(X^k)$ share the same limit points.
	\item $(Y^k)$ is precompact if and only if  $(f(Y^k))$ is bounded from below and $(X^k)$ is precompact.
\end{enumerate}
\end{proposition}

\begin{proof}
Item 1 comes directly from $\text{\textbf{HE}}_1 $. To prove item 2, we use Proposition \ref{P:hi}: from 
$\text{\textbf{H}}_1$ and $\text{\textbf{H}}_3(i) $ we have that
\begin{equation}\label{LA:err:1}
\underline{a} \Vert Y^{k+1} - Y^k \Vert^2 \leq f(Y^k) - f(Y^{k+1}),
\end{equation}
hence $(f(Y^k))$ is a decreasing sequence. Then we can sum inequality (\ref{LA:err:1}) to obtain that
\begin{equation}\label{LA:err:2}
\underline{a} \sum\limits_{k \in \N}^{} \Vert Y^{k+1} - Y^k \Vert^2 \leq f(Y^0) - \inf\limits_{\kin} f(Y^k) <+\infty.
\end{equation}
Since we have $\Vert r_i^k \Vert \leq \frac{\sigma}{2} \Vert y_i^{k+1} - y_i^k \Vert + \mu_k$ where $\mu_k \in \ell^1$ and $\Vert y_i^{k+1} - y_i^k \Vert  \leq \Vert Y^{k+1} - Y^k \Vert$ which is in $\ell^2$, we deduce that $\Vert r_i^k \Vert \in \ell^2$, and the same holds for $\Vert s_i^k \Vert$. So the errors converge to zero and $(X^k)$ and $(Y^k)$ have the same limit points. Item 3 follows from item 2 and the following: suppose that we have a  subsequence $(Y^{n_k})$ converging to some $Y^\infty=(y_1^\infty,...,y_p^\infty) \in \HH$. Since  $f$ is lower semi-continuous and $(f(Y^{k}))$ is decreasing, we have that $\inf\limits_{\kin} f(Y^k)$ is bounded from below by $f(Y^\infty)$.\bx
\end{proof}

 An other disadvantage to the direct application of Theorem \ref{T:Tmain} is that it asks the $f$-precompactness of $(Y^k)$. In some cases, precompactness of a sequence can be deduced using compact embeddings between Hilbert spaces. Sequences remaining in a sublevel set of an inf-compact function $f$ are also precompact. However, $f$-precompactness is harder to obtain without further continuity assumption on $f$. Actually, both limit and $f$-limit points coincide whenever the parameters are bounded:

\begin{proposition}\label{P:continuity}
If either $\beta_k \leq \bar \beta$ or $f$ is continuous on its domain, then $(Y^k)$ is $f$-precompact if and only if it is precompact.
\end{proposition}

\begin{proof}
Suppose that we have $Y^{k_n}$ converging to $Y^\infty$,  and  show that $f(Y^{k_n})$ converges also to $f(Y^\infty)$. Note that $f(Y^k)$ being decreasing and $f$ lower semicontinuous, we know that $Y^\infty$ must lie in the domain of $f$. If $f$ is continuous on its domain the conclusion is immediate. On the other hand suppose that $ \beta_k\leq \bar \beta$. Since $h$ is continuous, we only need to verify that $\lim\limits_{n \to +\infty} g_i(y_i^{k_n}) = g_i(y_i^\infty)$ for each $i\in \{1..p\}$. The lower-semicontinuity of $g_i$ already gives us
$g_i(y_i^\infty) \leq \liminf\limits_{n \to \infty} g_i(y_i^{k_n})$,
so we just have to prove that $\limsup\limits_{n \to \infty} g_i(y_i^{k_n}) \leq g_i(y_i^\infty)$, following the ideas of \cite{AttBolSva}.\\
Let $n\in\N^*$ and $k=k_n -1$, using the definition of the proximal operator, we have
\begin{align*}
 &  \ g_i (y_{i}^{k+1}) + \frac{1}{2} \Vert  y_i^{k+1} - y_i^k + A_{i,k}^{-1}\nabla_i h( Y_i^k + S_i^k) -r_i^k - s_i^k \Vert^2_{A_{i,k}} \\
  \leq & \ g_i (y_{i}^{\infty}) + \frac{1}{2} \Vert  y_i^\infty - y_i^k + A_{i,k}^{-1} \nabla_i h( Y_i^k + S_i^k) -r_i^k -s_i^k \Vert^2_{A_{i,k}},
\end{align*}
and the latter implies (using Cauchy-Schwartz and $\VERT A_{i,k} \VERT \leq \bar \beta$):
\begin{equation}\label{LA:err:4}
g_i (y_{i}^{k+1}) \leq g_i (y_{i}^{\infty}) + \frac{\bar \beta}{2} \Vert  y_i^{\infty} - y_i^k  \Vert^2 +\Vert y_i^{\infty} - y_i^{k+1}\Vert \left[ \Vert \nabla_i h(Y_i^k + S_i^k)\Vert + \bar \beta \Vert r_i^k + s_i^k \Vert \right].
\end{equation}
Now recall that $y_i^{k+1}=y_i^{k_n}$ tends to $y_i^\infty$ while $r_i^k + s_i^k$ goes to zero (see Proposition \ref{P:LimitPoints}). Observe also that $\nabla_i h(Y_i^k + S_i^k)$ is bounded  since it converges to $\nabla_i h(Y^\infty)$. Moreover, $\Vert  y_i^{\infty} - y_i^k  \Vert$ goes also to zero since we have
\begin{equation*}
\Vert  y_i^{\infty} - y_i^k  \Vert \leq \Vert  y_i^{\infty} - y_i^{k_n}  \Vert + \Vert  y_i^{k+1} - y_i^k  \Vert,
\end{equation*}
with $y_i^{k_n} \to y_i^{\infty} $ and $\Vert  y_i^{k+1} - y_i^k  \Vert \in \ell^2$ (see (\ref{LA:err:2})).
 Passing to the  upper limit in (\ref{LA:err:4}) leads finally to $\limsup\limits_{n\to +\infty} g_i (y_{i}^{k_n}) \leq g_i(y_i^\infty)$.
\bx
\end{proof}

 As a direct consequence of Propositions \ref{P:hi}, \ref{P:LimitPoints}, \ref{P:continuity} together with Theorem \ref{T:Tmain}, we finally get our convergence result for the AFB algorithm with errors. It extends the results of \cite{ChoPesRep14} (when taking a cyclic permutation on the variables) in two directions: the functions $g_i$ need not be continuous on their domain, or the step sizes can tend to 0.

\begin{theorem}\label{T:CVAFB}
Let $f$ be a K{\L} function. Let $(Y^k)$ be a precompact sequence generated by the AFB algorithm with errors, with \textbf{\em (HP)} and \textbf{\em (HE)} satisfied. Suppose that either $\beta_k$ remains bounded, or that $f$ is continuous on its domain. Hence, the sequence $(X^k)$ has finite length and converges toward a critical point of $f$.
\end{theorem}

\begin{remark}
{\em In the particular case where $S^k \equiv 0$, we know furthemore that the sequence $(X^k)$ is convergent with respect to $f$. This is no longer true in general if $f$ is not continuous and $S^k \neq 0$. As a simple counterexample, take $f : x \in \R \mapsto \vert x \vert_0 \in \R$ where $\vert x \vert_0=0$ if $x=0$, $\vert x \vert_0=1$ else.  By taking as parameters $A_k \equiv 2 id$, $r^k \equiv 0$, $s^k = \frac{1}{k}$ and $x^0=0$, it is easy to see, after applying the AFB algorithm, that $f(y^k) \equiv 0$ but $f( x^k ) \equiv 1$.}
\end{remark}

An analog of the capture result in Theorem \ref{T:localCV} can also be deduced:
 
\begin{theorem}\label{T:CaptureAFB}
Suppose that the K{\L} property holds in a global minimum $X^*$ of $f$. Let $(X^k)$ be a sequence generated by  the AFB algorithm with errors, satisfying \textbf{\em (HP)} and \textbf{\em (HE)} with $\mu_k \equiv 0$. Hence, there exist $\gamma>0$ and $\eta>0$ such that if $X^0 \in \underline{\Gamma}_\eta(X^*,\gamma)$, then $(X^k)$ has finite length and converges to a global minimum  of $f$.
\end{theorem}

\noindent To prove this theorem, it suffices to use $Y^0=X^0$, and to see at the end of the proof of Proposition \ref{P:hi} that $\mu_k =0$ iff $\epsilon_k = 0$, where $\epsilon_k$ is the parameter involved in $\hthree$. Then, apply Theorem \ref{T:localCV} together with Propositions \ref{P:hi} and \ref{P:LimitPoints}.

\subsection{Variable metric: towards generalized Newton methods}\label{SS:newton}

We focus here on the problem of minimizing a $C^{1,1}$ function $h : \R^N \rightarrow \R$ over a closed nonempty set $C \subset \R^N$.
The AFB algorithm reduces in this case to a projected-gradient  method, and allow us to compute in the explicit step a descent direction governed by a chosen metric $A_k$. As an example,  take $h(x) = \frac{1}{2}\langle Ax,x \rangle - \langle b,x \rangle$ with  $A \in \Sym(\R^N)$. In the unconstrained case, the Newton method (that is taking $A_k\equiv A$) is known to solve in one single step the problem. If we add a constraint $C$ it is easy to see that the Newton-projected method
\begin{equation}\label{D:NewtonProj}
x^{k+1} \in \proj_C^{A_k} \ \left(x^k - A_k^{-1} \nabla h(x^k) \right)
\end{equation}
gives the minimum of $h$ over $C$ in one single step. For a general function $h$,  (\ref{D:NewtonProj}) reduces to the minimisation over $C$ of a quadratic model of $h$, as stressed in (\ref{D:FBmetric2}). One can see on this example that computing the proximal operator relatively to the metric $A_n$ used in the explicit step (and not the ambient metric !) is of crucial importance in this method.

The spirit here is to use second-order information from $h$ in order to improve the convergence of the method. In the unconstrained case, a popular choice of metric is given by Newton-like methods, where the metric at step $k$ is induced by (an approximation of) the Hessian $\nabla^2 h(x^k)$. Since it is often impossible to know in advance whether or not the Hessian is uniformly elliptic at each $x^k$, a positive definite approximation has to be chosen.

\noindent We detail here a natural way to chose this positive definite $A_k \sim \nabla^2 h(x^k)$ in closed loop, and show that this method remains in the setting of Theorem \ref{T:CVAFB}. Since it generalizes the \textit{Levenberg-Marquardt} method used in the convex case (see \cite{AttSva11}) we will refer to the \textit{Generalized Levenberg-Marquardt} method for this way of designing $A_k$. One of the interesting aspect of the method is that such a matrix can be defined even if $h$ is only $C^{1,1}$ and not $C^2$, since the differentiability of $\nabla h$ is not necessary in Theorem \ref{T:CVAFB}. Another interesting aspect is that the splitting approach led us to solve constrained minimization problems with a Newton-projected approach.

We set $\sym(\R^N)$ the closed convex cone of nonnegative matrices. Consider the generalized Hessian of $h$, by taking the generalized Jacobian of $\nabla h$ in sense of Clarke. Given $x \in \R^N$ it is $$\partial^2 h(x):= co \{ \lim\limits_{n \to +\infty} \nabla^2 h(x_n), \text{ where } \nabla h \text{ is differentiable at } x_n \text{ and } x_n \to x \}.$$
This set contains symmetric matrices bearing second-order information on $h$. Hence, the Generalized Levenberg-Marquardt method to compute $A_k \in \Sym (\R^N)$ from a given $x^k \in \R^N$ is the following : for $\varepsilon >0$,
$$\begin{array}{l}
\text{Take } H_k \in \partial^2 h(x^k),  \\
\text{Project } P_k = \proj_{\sym (\R^N)} (H_k), \\
\text{Regularize } A_k = P_k + \varepsilon I_N.
\end{array}$$
A globalized version of the method can be considered by taking step sizes ensuring descent. Then the following convergence result holds:

\begin{proposition}
Let $f(x):=h(x) + \delta_C (x)$ be a K{\L} function, where $C \subset \R^N$ is closed nonempty and $h$ is differentiable with a $L$-Lipschitz gradient. Let $x^0 \in H$ and suppose that $(x^k)$ is a bounded sequence generated by
$$ x^{k+1} \in \proj_C^{A_k}\left(x^k - \lambda_k A_k^{-1} \nabla h(x^k)\right),$$
where $A_k$ is selected with the Generalized Levenberg-Marquardt process detailed above, and the stepsizes $\lambda_k$ satisfy:
$$ 0 < \lambda_k \leq \bar{\lambda} < \frac{\varepsilon}{L}, \quad \lambda_k \notin \ell^1 \ \text{ and } \ \sup\limits_{k\in\N} \frac{\lambda_{k+1}}{\lambda_k} < + \infty.$$
Then the sequence has finite length and is converging to a critical point of $f$.
\end{proposition}

\begin{proof}
Start by observing that $\proj_C^{A_k} = \proj_C^{\lambda_k^{-1} A_k}$, so the algorithm falls in the setting of the AFB algorithm. According with the previous notations,  $\nabla h$ being $L$-Lipschitz continuous implies that the sequence $(\VERT H_k \VERT )$ is bounded by $L$, and so $(\VERT P_k \VERT)$ remains bounded by $2L$. 
To conclude through Theorem \ref{T:CVAFB} we just need to check the hypotheses \textbf{(HP)} on the parameters $\frac{1}{\lambda_k} A_k$. We have here $\alpha_k=\alpha(\frac{1}{\lambda_k} A_k) \geq \varepsilon {\lambda_k}^{-1} \geq \varepsilon \bar{\lambda}^{-1} > L$ and $
\beta_k = \VERT \frac{1}{\lambda_k} A^k \VERT \leq  (2L + \epsilon)\lambda_k^{-1}.$
Thus $\mathbf{HP}_1$ is satisfied, while items $\mathbf{HP}_2$ and $\mathbf{HP}_3$ follows directly from  the hypotheses made on $(\lambda_k)$. Since the indicator function $\delta_C$ is continuous on its domain, the hypotheses of Theorem \ref{T:CVAFB} are satisfied.\bx
\end{proof}

This extends, in a way, results from the convex setting to the nonconvex one, enforcing moreover the strong convergence (see \cite[Theorem 7.1]{AttSva11}).
	
A drawback of this method is that the Hessian increases the complexity of implementation since a matrix must be inverted in the explicit step. An alternative is the Broyden-Fletcher-Goldfarb-Shanno (BFGS) update scheme (see \cite{Ber},\cite{SraNowWri}), using only first-order information to compute the inverse of the Hessian. On the other hand, the implicit step gains also in complexity since one must project onto a constraint relatively to a given metric, which is nontrivial even for simple constraints. For linear constraints, a particular second-order model of the Hessian can be taken in order to reduce the implicit step in a trivial orthogonal projection step (see \cite{SraNowWri,GafBer,Ber82}).

Newton-like methods are expected to have good convergence rates in exchange for a more expensive implementation. An interesting question is whether one can obtain convergence rates beyond the results in Subsection \ref{SS:rates}, by exploiting, not only the K{\L} nature of the function, but also the specific properties of the matrices selected by the Generalized Levenberg-Marquardt process.

%%%%%%%%%%%%%%%%%%%%%%%%%%%%%%%%%%%%%%%%%%%%%%%%%%%%%%%%%%%%%%%%%%%%%%%%%%%%%%%%%%%%%%%%%%%%%%%%%%%%%%%%%%%%%%%%

\section{Applications}\label{S:applications}

The framework presented in this paper is suitable for the numerical resolution of a wide variety of structured problems. Consider for instance the problems arising in image processing and data compression, which are generally semi-algebraic by nature \cite{Don06,DonTan10,DauDefDem04}. Indeed, they generally involve the semi-algebraic \textit{counting norm} $\Vert x \Vert_0 := \sharp \{ i \ | \ x_i \neq 0 \}$, whose proximal operator (the hard shrinkage operator, see \cite{AttBolSva}) is  easily implementable.
 Feasability problems with semi-algebraic (eventually nonconvex) constraints are also well suited for the AFB method (see \cite{LewLukMal09,AttBolSva}). The search for equilibria of nonlinear partial differential equations has already been tackled using the K\L \ inequality \cite{MerPie10}. It should now be improved by using splitting methods more adapted to the structure of the problem. Let us end by discussing in some detail the {\em sparse and low-rank matrix decomposition}, for which the AFB method is particularly well adapted, in view of its structure.\\

\noindent{\bf Sparse and low-rank matrix decomposition.} The problem of recovering the sparse and low-rank components of a matrix arises naturally in various areas such as model selection in statistics or system identification in engineering (see \cite{ChaSanParWil} and references therein). Denote by $\Vert X \Vert_0$ the number of nonzero components of $X \in \mathcal{M}_{m,n} (\R)$. Given $A \in \mathcal{M}_{m,n} (\R)$ and bounds $r,s \in \mathbb{N}$, the low-rank sparse matrix decomposition problem consists in finding $X,Y \in \mathcal{M}_{m,n}(\R)$ such that $A=X+Y$ with $\rank (x) \leq r$ and $\Vert Y \Vert_0 \leq s$. Endowing $\mathcal{M}_{m,n} (\R)$ with the Frobenius norm, this reduces to
$$ \underset{X,Y \in  \mathcal{M}_{m,n}(\R)}{\text{minimize}} \  \delta_{\{\rank \cdot \leq r\}} (X) + \delta_{\{\Vert \cdot \Vert_0 \leq s\}} (Y) + \frac{1}{2}\Vert A-X-Y\Vert_F^2.$$

\noindent An approach to solve this problem consists in doing a convex relaxation of the objective function (see \cite{Ganesh,YuaYan}). The sparsity and low-rank properties are obtained by minimizing the $\ell^1$ and  nuclear norms, respectively (see \cite{RecFazPar}). 

\noindent The K{\L} framework is well adapted to the original  nonconvex (but semialgebraic!) problem and offers convergent numerical methods. Moreover, the AFB method is well suited for its structure in separated variables involving smooth and nonsmooth parts. It leads to an \textit{Alternating Averaged Projected Method}: given $(X_0,Y_0)$, take $(\lambda_k)$, $(\mu_k)$ with $0<\underline{\tau}\le\lambda_k,\ \mu_k\le\bar \tau<1 $. For $k\ge 0$, define
$$\begin{array}{rcl}
X^{k+1} &\in& \proj_{\{\rank \cdot \leq r\}} (\lambda_k (A-Y^k) + (1-\lambda_k ) X^k),\\
Y^{k+1} &\in& \proj_{\{\Vert \cdot \Vert_0 \leq s \}}(\mu_k (A-X^{k+1}) + (1-\mu_k) Y^{k}).
\end{array}$$
Projection onto $\{\rank \cdot \leq r\}$ can be done using the Singular Value Decomposition (see Eckart-Young's Theorem). To project onto $\{\Vert \cdot \Vert_0 \leq s \}$, one simply sets all the coefficients to zero, except for the $s$ largest ones (in absolute value). Theorem \ref{T:CaptureAFB} guarantees convergence to the solution for sufficiently close initialization. This example illustrate the discussion in Remark \ref{R:Lipschitz}: here we have $L=1$, while if one consider the Lipschitz constant of the gradient of $(X,Y) \mapsto \frac{1}{2}\Vert A-X-Y\Vert_F^2$, we would have had $L=2$, that is a strictly smaller upper bound for the parameters.

\section{Concluding Remarks}\label{SS:conclusion}

We have given a unified way to handle various recent descent algorithms, and derived general convergence rate results in the K{\L} framework. These are applicable to potential future numerical methods. Some improvements have been explored, and a novel projected Newton-like method has been proposed.

A challenging task is to extend the present convergence analysis to algorithms that do not satisfy the sufficient decrease condition $\hone$. This will allow to consider acceleration schemes like the ones studied in \cite{Nest2,BT09,BT}, or primal-dual methods based on a Lagrangian approach. A recent preprint \cite{ipiano} seems to be an interesting first attempt in this direction.

From the applications point of view, the {\em counting norm} $\|\cdot\|_0$ evoked in Section \ref{S:applications} has a natural extension to an infinite-dimensional functional setting, namely the measure of the support of a function $u$ defined on some $\Omega\subset\R^N$. An interesting but challenging issue is to apply our algorithm to this extension in order to solve the problem of sparse-optimal control of partial differential equations. From the implementation point of view, it suffices to apply the one-dimensional hard shrinkage operator at each point. Nevertheless, the verification of the K{\L} inequality for this function has not been established and will probably rely on sophisticated arguments concerning the geometry of Hilbert spaces. Then, there is the natural question whether this approach is more efficient than those using the $L^1$ norm (see, for instance, \cite{Cas_Her_Wac}).

Finally, it is worth mentioning that the results in Section \ref{S:Abstract} remain true in the more general context of a normed space, adapting the definition of {\em subdifferential} and {\em lazy slope} in an obvious manner.

%%%%%%%%%%%%%%%%%%%%%%%%%%%%%%%%%%%%%%%%%%%%%%%%%%%%%%%%%%%%%%%%%%%%%%%%%%%%%%%%%%%%%%%%%%%%%%%%%%%%%%%%%%%%%%%%

\

\noindent \textit{Acknowledgements} : The authors thank H. Attouch for useful remarks. They would also like to thank the anonymous reviewer for his careful reading and constructive comments.

%%%%%%%%%%%%%%%%%%%%%%%%%%%%%%%%%%%%%%%%%%%%%%%%%%%%%%%%%%%%%%%%%%%%%%%%%%%%%%%%%%%%%%%%%%%%%%%%%%

\begin{appendix}

\section{Appendix}

\subsection{Proofs of Theorems \ref{T:Tmain} and \ref{T:localCV}}\label{A:1}

The argument is a straightforward adaptation of the ideas in the proof of \cite[Lemma 2.6]{AttBolSva}. One first proves:

\begin{lemma}\label{L:L1}
Let $\hone$ and $\htwo$ hold and fix $k\in\N$. If $x^k$ and $x^{k+1}$ belong to $\underline{\Gamma}_\eta(x^*,\delta)$,
then
\begin{equation}\label{E:L1}
2\norm{x^{k+1}-x^k} \leq \norm{x^k-x^{k-1}}+\frac{1}{a_kb_k}\big[\varphi(f(x^k)-f(x^*))-\varphi(f(x^{k+1})-f(x^*))\big]+\epsilon_{k}.
\end{equation}
\end{lemma}

For the next results, we introduce the following auxiliary property (automatically fulfilled under the hypotheses of Theorems \ref{T:Tmain} and \ref{T:localCV}), which includes a {\em stability} of the sequence $(x^k)_{k\in\N}$ with respect to the point $x^*$, along with a sufficiently close initialization.\\

\noindent\underline{$\mathbf{S}(x^*,\delta,\rho)$:} There exist $\delta >\rho>0$ such that
\begin{itemize}
  	\item[i)] For each $k\in\N$, if $\frac{}{}\!x^0,\dots,x^k\in \underline{\Gamma}_\eta(x^*,\rho)$, then $x^{k+1}\in\underline{\Gamma}_\eta(x^*,\delta)$;
  	\item[ii)] The initial point $x^0$ belongs to $\Gamma_\eta(x^*,\rho)$ and
\begin{equation}\label{E:H4_h2}
  	  \norm{x^*-x^{0}}+2\sqrt{\frac{f(x^{0})-f(x^*)}{a_0}}+ M\varphi(f(x^{0})-f(x^*))+\sum_{i=1}^{+\infty}\epsilon_i < \rho.
\end{equation}
\end{itemize}

Then, we have the following estimation:

\begin{lemma}\label{L:L2}
Let $\hone$, $\htwo$, $\hthree$ and $\mathbf{S}(x^*,\delta,\rho)$ hold, and note $M=\sup_{k\in\N^*}\frac{1}{a_kb_k}<+\infty$. Then, for all $K\in\N^*$, we have $x^K\in \underline{\Gamma}_\eta(x^*,\rho)$ and
$$\sum_{k=1}^{K}\norm{x^{k+1}-x^k}+\norm{x^{K+1}-x^K}\leq  \norm{x^1-x^0}+M\big[\varphi(f(x^1)-f( x^* ))-\varphi(f(x^{K+1})-f(x^*))\big]+\sum_{k=1}^{K}\epsilon_k.$$
\end{lemma}

The basic asymptotic properties are given by the following result:

\begin{proposition}\label{P:Pmain}
Let $\hone$, $\htwo$, $\hthree$ and $\mathbf{S}(x^*,\delta,\rho)$ hold. Then $x^k\in\underline{\Gamma}_\eta(x^*,\rho)$ for all $k$ and converges to some $\overline x$ lying in the closed ball $\overline{B(x^*,\rho)}$. Moreover $\sum_{k=1}^\infty\norm{x^{k+1}-x^k}<\infty$, $\liminf_{k\to\infty}\Vert \partial f(x^k) \Vert_- =0$, and $f(\overline x)\le\lim_{k\to\infty}f(x^k)=f(x^*)$.
\end{proposition}

\begin{proof}
Capture, convergence and finite length follow from Lemma \ref{L:L2} and $\hthree$. Next, since $(b_k)\notin\ell^1$ and $\sum_{k=1}^\infty b_{k+1}\Vert \partial f(x^k) \Vert_- \le \sum_{k=1}^\infty\norm{x^{k+1}-x^k}+\sum_{k=1}^\infty\eps_{k+1}<\infty$, we obtain $\liminf_{k\to\infty}\Vert \partial f(x^k) \Vert_- =0$. Finally, observe that that $\lim_{k\to\infty}f(x^k)$ exists because $f(x^k)$ is decreasing and bounded from below by $f(x^*)$ and the lower-semicontinuity of $f$ implies $f(\overline x)\le \lim_{k\to\infty}f(x^k)$. If $\lim_{k\to\infty}f(x^k)=\beta>f(x^*)$, the K{\L} inequality and the fact that $\varphi'$ is decreasing imply
$\varphi'(\beta-f(x^*))\Vert \partial f(x^k) \Vert_-  \ge \varphi'(f(x^k)-f(x^*))\Vert \partial f(x^k) \Vert_- \ge 1$
for all $k\in\N$, which is impossible because $\liminf_{k\to\infty}\Vert \partial f(x^k) \Vert_- =0$. Whence $\beta = f(x^*)$.\bx
\end{proof}

We are now in position to complete the proofs of Theorems \ref{T:Tmain} and \ref{T:localCV}.

\noindent\textit{Proof of Theorem \ref{T:Tmain}}\,\ Let $x^{n_k}\to x^*$ with $f(x^{n_k})\to f(x^*)$ as $k\to\infty$. Since $f(x^k)$ is nonincreasing and admits a limit point, we deduce  that $f(x^k)\downarrow f(x^*)$. In particular, we have $f(x^*)\leq f(x^k)$ for all $k\in \N$. The function $f$ satisfies the K{\L} inequality on $\Gamma_\eta(x^*,\delta)$ with desingularizing function $\varphi$. Let $K_0\in\N$ be sufficiently large so that $f(x^{K})-f(x^*)<\min\{\eta,\underline{a}\delta^2\}$, and pick $\rho>0$ such that $f(x^K)-f(x^*)<\underline{a}(\delta-\rho)^2$. Hence, $f(x^*)\leq f(x^{k+1})<f(x^*)+\eta$ for all $k\ge K$ and
$$\norm{x^{k+1}-x^k}\le\sqrt{\frac{f(x^k)-f(x^{k+1})}{a_k}}\leq \sqrt{\frac{f(x^K)-f(x^*)}{\underline{a}}}< \delta-\rho,$$
which implies part i) of $\mathbf{S}(x^*,\delta,\rho)$. Now take $K\ge K_0$ such that
$$\norm{x^*-x^{K}}+2\sqrt{\frac{f(x^{K})-f(x^*)}{a_{n_K}}}+M\varphi(f(x^{K})-f(x^*))+\sum_{k=K+1}^{+\infty}\epsilon_{k}<\rho.$$
The sequence $(y^k)_{k\in \N}$ defined by $y^k=x^{K+k}$ for all $k\in\N$ satisfies the hypotheses of Proposition \ref{P:Pmain}. Finally, since the whole sequence $(y^k)_{k\in\N}$ is $f$-convergent toward $x^*$ and $\liminf_{k\to\infty}\Vert \partial f(y^k) \Vert_-=0$, we conclude that $x^*$ must be critical using Lemma \ref{LemCritical}.\bx

\bigskip

\noindent\textit{Proof of Theorem \ref{T:localCV}}\,\ Since $f$ has the K{\L} property in $x^*$, there is a strict local upper level set $\Gamma_{\eta}(x^*,\delta)$ where the K{\L} inequality holds with $\varphi$ as a desingularizing function. Take $\rho<\frac{3}{4}\delta$ and then $ \gamma<\frac{1}{3}\rho$. If necessary, shrink $\eta$ so that
$2\sqrt{\dfrac{\eta}{\underline{a}}} + M \varphi (\eta) < \dfrac{2\rho}{3}$. This is possible since $\varphi$ is continuous in $0$ with $\varphi(0)=0$. Let $x^0\in \underline{\Gamma}_\eta(x^*,\gamma)\subset \underline{\Gamma}_\eta(x^*,\rho)$. It suffices to verify that $\mathbf{S}(x^*,\delta,\rho)$ is fulfilled and use Proposition \ref{P:Pmain}. For i), let us suppose that $x^0,\dots,x^k$ lie in $\underline{\Gamma}_\eta(x^*,\rho)$ and prove that $x^{k+1} \in \underline{\Gamma}_\eta(x^*,\delta)$. Since $x^*$ is a global minimum, from $\hone$ and the fact that $\left(f(x^k)\right)_{k\in\mathbb{N}}$ is decreasing, we have
$$f(x^*)+\underline{a} \Vert x^{k+1}-x^k \Vert^2\le f(x^{k+1})+\underline{a} \Vert x^{k+1}-x^k \Vert^2\le f(x^k)\le f(x^0)< f(x^*) + \eta.$$
It follows that $\|x^{k+1}-x^*\|\le\|x^{k+1}-x^k\|+\|x^k-x^*\|<\sqrt{\frac{\eta}{\underline{a}}}+\rho<\frac{4}{3}\rho<\delta$, and so $x^{k+1}\in \underline{\Gamma}_\eta(x^*,\delta)$. Finally, we have
$$\|x^0-x^*\|+2\sqrt{\frac{f(x^0)-f(x^*)}{a_0}}+M\varphi(f(x^0)-f(x^*))<\frac{1}{3}\rho+2\sqrt{\dfrac{\eta}{\underline{a}}} + M \varphi (\eta) <\rho,$$
which is precisely ii).\bx

\subsection{Proof of Proposition \ref{P:hi}}\label{Ap:2}

\begin{proof}
Since $X_i^k=Y_i^k +S_i^k$, we can rewrite the algorithm as
\begin{equation}\label{LA:h1:1}
y_i^{k+1}  \in  \prox^{A_{i,k}}_{g_i} (y_i^k - A_{i,k}^{-1} \nabla_i h(Y_i^k + S_i^k) + r_i^k + s_i^k).
\end{equation}
We start by showing that $\hone$ is satisfied.

Let $i=1..p$ be fixed. Using the definition of the proximal operator $\prox_{g_i}^{A_{i,k}}$ in (\ref{LA:h1:1}) and developing the squared norms gives
\begin{align}
& \ g_i (y_{i}^k) - g_i (y_{i}^{k+1}) \label{LA:h1:2} \\
%\geq & \ \frac{1}{2}  \Vert  y_i^{k+1} - y_i^k\Vert_{A_{i,k}}^2 +  \langle y_i^{k+1} - y_i^k,  A_{i,k}^{-1} \nabla_i h(Y_i^k + S_i^k)-r_i^k - s_i^k \rangle_{A_{i,k}} \nonumber \\
 \geq & \  \frac{1}{2}  \Vert  y_i^{k+1} - y_i^k\Vert_{A_{i,k}}^2 + \langle y_i^{k+1} - y_i^k,   \nabla_i h(Y_i^k + S_i^k)\rangle - \langle y_i^{k+1} - y_i^{k}, r_i^k +s_i^k\rangle_{A_{i,k}}.\nonumber
\end{align}
Using $\mathbf{HE}_3$ in (\ref{LA:h1:2}), the latter results in
\begin{equation}\label{LA:h1:3}
\ g_i (y_{i}^k) - g_i (y_{i}^{k+1}) \geq  \frac{1}{2}  \Vert  y_i^{k+1} - y_i^k\Vert_{\rho A_{i,k}}^2 + \langle y_i^{k+1} - y_i^k,   \nabla_i h(Y_i^k + S_i^k)\rangle.
\end{equation}
For fixed $k\in \N$ and $i=1,\dots,p$, introduce the function 
\begin{equation}\label{LA:h1:4}
\tilde h_{i,k} : y_i \in H_i \mapsto (y_1^{k+1},..,y_{i-1}^{k+1},y_i,y_{i+1}^k,..,y_p^k) \in \R
\end{equation}
which satisfies $\tilde h_{i,k} (y_i^k) = h(Y_i^k), \ \tilde h_{i,k} (y_i^{k+1}) = h(Y_{i+1}^k) \text{ and } \nabla \tilde h_{i,k} (y_i^k) = \nabla_i h (Y_i^{k}).$
Applying the descent lemma to $\tilde h_{i,k}$, we obtain
\begin{equation}\label{LA:h1:6}
h(Y_{i+1}^k) - h(Y_i^k) - \langle y_i^{k+1} - y_i^k , \nabla_i h(Y_i^k) \rangle \leq \frac{L}{2} \Vert y_i^{k+1} - y_i^k \Vert^2.
\end{equation}
Then, combining (\ref{LA:h1:3}) and (\ref{LA:h1:6}) we get
\begin{align}
& \ g_i (y_{i}^k) - g_i (y_{i}^{k+1}) + h(Y_i^k) - h(Y_{i+1}^k) \label{LA:h1:7} \\
 \geq & \  \frac{1}{2}  \Vert  y_i^{k+1} - y_i^k\Vert_{\rho A_{i,k} - Lid_{H_i}}^2 + \langle y_i^{k+1} - y_i^k,   \nabla_i h(Y_i^k + S_i^k) - \nabla_i h(Y_i^k) \rangle,\nonumber
\end{align}
where $\rho A_{i,k} - Lid_{H_i}$ remains coercive, since $\rho \alpha_k > L$. Using successively the Cauchy-Schwartz inequality, the Lipschitz property of $\nabla_i h$ (see Remark \ref{R:Lipschitz}) and $\mathbf{HE}_1$, one gets
\begin{eqnarray*}
\langle y_i^{k+1} - y_i^k,   \nabla_i h(Y_i^k + S_i^k) - \nabla_i h(Y_i^k) \rangle &
\geq & \ - \Vert y_i^{k+1} - y_i^k \Vert \Vert \nabla_i h(Y_i^k +S_i^k ) - \nabla_i h(Y_i^k) \Vert\\
& \geq & \ -L \ \Vert y_i^{k+1} - y_i^k \Vert \Vert S_i^k \Vert  \geq \ - \frac{\sigma L}{2\sqrt{p}} \Vert y_i^{k+1} - y_i^k \Vert^2.
\end{eqnarray*}
Inserting this estimation in (\ref{LA:h1:7}) we deduce that
\begin{equation}\label{LA:h1:8}
 g_i (y_{i}^k) - g_i (y_{i}^{k+1}) + h(Y_i^k) - h(Y_{i+1}^k) \geq  \  \frac{1}{2}  \Vert  y_i^{k+1} - y_i^k\Vert_{\rho A_{i,k} - L(\sigma  \frac{\sqrt{p}}{p} +1)id_{H_i}}^2.
\end{equation}
We can now conclude by summing all these inequalities for $i=1,\dots,p$:
\begin{align}
f(Y^k) - f(Y^{k+1}) = & \ \sum\limits_{i=1}^{p} g_i(y_i^k) -g_i(y_i^{k+1}) + h(Y_i^k) - h(Y_{i+1}^k) \label{LA:h1:9}\\
\geq & \ \frac{1}{2} \sum\limits_{i=1}^{p} \Vert  y_i^{k+1} - y_i^k\Vert_{\rho A_{i,k} - L({\sigma}\frac{\sqrt{p}}{p} +1)id_{H_i}}^2 \nonumber
\end{align}
so $\hone$ is fulfilled with $a_k=\frac{\rho \alpha_k - L(\sigma \frac{\sqrt{p}}{p} +1)}{2}$.\\
\indent To prove $\htwo$, fix $i=1,\dots,p$ and use Fermat's first order condition in (\ref{LA:h1:1}) to get:
\begin{equation}\label{LA:h2:1}
0 \in \partial g_i(y_i^{k+1}) + \left\{ A_{i,k}(y_i^{k+1} - y_i^k) - A_{i,k}(r_i^k + s_i^k) + \nabla_i h(Y_i^k + S_i^k) \right \}
\end{equation}
Define $w_i^{k+1} := \nabla_i h(Y^k) - \nabla_i h(Y_i^k + S_i^k) - A_{i,k}(y_i^{k+1} - y_i^k) + A_{i,k}(r_i^k + s_i^k)$ which lies in $\partial g_i(y_i^{k+1}) + \nabla_i h(Y^{k+1})$,  by (\ref{LA:h2:1}). The triangle inequality gives
\begin{equation}\label{LA:h2:2}
\Vert w_i^{k+1} \Vert \leq  \ \beta_k \left( \Vert y_i^{k+1} - y_i^k \Vert + \Vert r_i^k \Vert +  \Vert s_i^k \Vert \right)  + \Vert \nabla_i h(Y_i^k + S_i^k) - \nabla_i h(Y^{k+1}) \Vert,
\end{equation}
where we use the error estimations from \textbf{(HE)}
\begin{equation}\label{LA:h2:3}
\Vert r_i^k \Vert +  \Vert s_i^k \Vert \leq \ \sigma \Vert y_i^{k+1} - y_i^k \Vert + \mu_k,
\end{equation}
and the $\sqrt{p} L$-Lipschitz continuity of $\nabla_i h$:
\begin{align}
\Vert \nabla_i h(Y_i^k + S_i^k) - \nabla_i h(Y_i^k) \Vert \leq & \ \sqrt{p} L \Vert Y_i^k - Y^{k+1} + S_i^k \Vert \label{LA:h2:4} \\
 \leq  & \ \sqrt{p} L \Vert Y^{k+1} - Y^k \Vert + {\sqrt{p} L \sigma} \Vert y_i^{k+1} - y_i^k \Vert \nonumber.
\end{align}
Combining (\ref{LA:h2:2}), (\ref{LA:h2:3}) and (\ref{LA:h2:4}) leads to
\begin{equation}\label{LA:h2:5}
\Vert w_i^{k+1} \Vert \leq  \ (\beta_k (1 + \sigma ) + \sqrt{p}L\sigma ) \Vert y_i^{k+1} - y_i^k \Vert + \sqrt{p}L \Vert Y^{k+1} - Y^k \Vert + \beta_k \mu_k.
\end{equation}
Define now $W^{k+1} := (w_1^{k+1},...,w_p^{k+1}) \in \partial f (Y^{k+1})$ (recall the definition of $w_i^{k+1}$). Then through the sum over $i=1..p$ of inequality (\ref{LA:h2:5}) we have (using $\sqrt{p} \leq p \leq p^2$)
\begin{align*}
\Vert W^{k+1} \Vert \ \leq  \ \sum\limits_{i=1}^{p} \Vert w_i^{k+1} \Vert \ \leq \ p\beta_k \mu_k + p^2 (\beta_k+L) (1 + \sigma )\Vert Y^{k+1} - Y^k \Vert.
\end{align*}
Hence $\htwo$ is verified with $b_{k+1}=\frac{1}{p^2 (1+\sigma)(\beta_k +L)}$ and $\epsilon_{k+1}= \frac{\beta_k \mu_k}{p(1+\sigma)(\beta_k +L)}$.

Now we just need to check that the hypotheses $\hthree$ are satisfied with our hypotheses on $\alpha_k$, $\beta_k$ and $\mu_k$. Clearly $\hthree (i)$ holds since we've supposed that $\alpha_k \geq \underline{\alpha} > (\sigma \frac{\sqrt{p}}{p} +1) \frac{L}{\rho}$. Then $\hthree (ii)$ asks that $b_k \notin \ell^1$, which is equivalent to $\frac{1}{\beta_k +L} \notin \ell^1$ in our context. This holds since we've supposed that $\frac{1}{\beta_k} \notin \ell^1$. 
Hypothese $\hthree (iii)$ is satisfied because $\frac{\beta_k}{\alpha_{k+1}}$ is supposed to be bounded.
Finally, $\hthree (iv)$ asks the summability of $\frac{\beta_k \mu_k}{\beta_k +L}$ which is bounded by $\mu_k \in \ell^1$.\bx
\end{proof}

\end{appendix}



\end{document}